\documentclass[reqno,11pt]{amsart}        
\usepackage{amsfonts,amsmath,amssymb,amsthm,colordvi,mathrsfs,graphicx}
\usepackage{caption,subcaption,float}
\usepackage[utf8]{inputenc}
\usepackage{mathrsfs}
\usepackage{geometry}
\newtheorem{conjecture}{Conjecture}

\usepackage{enumerate}

\usepackage{amsmath, amssymb, amsthm, geometry, hyperref}

\usepackage{tabularx}
\usepackage{tabulary}
 
\hypersetup{
  colorlinks=true,
  linkcolor=blue,
  citecolor=blue,
  urlcolor=blue
}
\usepackage[all,knot,poly]{xy}
\usepackage{tikz-cd}
\usepackage{tikz}
\usepackage{verbatim}
\usetikzlibrary{arrows}
\usetikzlibrary{knots}
\tikzset{every path/.style={line width=0.4pt},every node/.style={transform shape,knot crossing,inner sep=1.5pt},>=triangle 60,text node/.style={rectangle,transform shape=false,black}}

 \usetikzlibrary{decorations.markings, arrows.meta}

\theoremstyle{plain}      
\newtheorem{thm}{Theorem}[section]     
\newtheorem{theorem}[thm]{\bf Theorem}     
\newtheorem{corollary}[thm]{\bf Corollary}

\theoremstyle{remark}

\newtheorem{remark}[thm]{Remark} 
\newtheorem{claim}[thm]{Claim}     
\theoremstyle{definition}      
\newtheorem{definition}[thm]{Definition}     

\newcommand{\di}{\displaystyle}

\newcommand{\Vol}{\mathop{\rm Vol}\nolimits}

\newcommand{\Log}{\mathop{\rm Log}\nolimits}


\begin{document}

\title{TBA}

\author{Mounir Nisse}

\address{Mounir Nisse\\
Department of Mathematics, Xiamen University Malaysia, Jalan Sunsuria, Bandar Sunsuria, 43900, Sepang, Selangor, Malaysia.
}
\email{mounir.nisse@gmail.com, mounir.nisse@xmu.edu.my}
\thanks{}
\thanks{This research is supported in part by Xiamen University Malaysia Research Fund (Grant no. XMUMRF/ 2020-C5/IMAT/0013).}

 \title{Degree Bounds for Amoeba Contours}

\maketitle

 \begin{abstract}
We investigate the real algebraic complexity of contours of amoebas associated with algebraic hypersurfaces and complete intersections in complex algebraic tori. Motivated by the foundational estimates of Lang--Shapiro--Shustin \cite{LSS}, we develop a toric and logarithmic approach relating contour geometry to logarithmic Gauss maps, mixed volumes, and Newton polytope geometry. We prove that the logarithmic criticality equations admit substantially sharper asymptotic behavior than the classical Pfaffian bounds, reducing the expected growth from order $d^{2n}$ to order $d^n$. We further formulate a conjectural asymptotic theory for complete intersections based on logarithmic Grassmannian geometry and Schubert degeneracy loci. The work establishes new connections between amoeba theory, real algebraic geometry, tropical geometry, and logarithmic toric geometry.
\end{abstract}

\section*{Introduction}

The theory of amoebas has become one of the central themes connecting algebraic geometry, tropical geometry, convex geometry, and real algebraic geometry. Since the pioneering work of Gelfand, Kapranov, and Zelevinsky \cite{GKZ94}, and the later foundational developments of Passare and Rullg{\aa}rd \cite{PR04}, Forsberg--Passare--Tsikh \cite{FPT00}, Mikhalkin \cite{Mikhalkin2004}, and Maclagan and Sturmfels \cite{MaclaganSturmfels}, amoebas have emerged as fundamental logarithmic shadows of algebraic varieties inside complex algebraic tori. Their geometry encodes subtle interactions between algebraic equations, Newton polytopes, tropical degenerations, logarithmic structures, and convex geometry.

The interaction between amoebas and tropical geometry has become especially important following the development of tropical methods in algebraic geometry and nonarchimedean geometry \cite{Kapranov,Mikhalkin2004}. In particular, tropical varieties arise as asymptotic degenerations of amoebas, while Newton polytopes govern both the tropical and logarithmic behavior of algebraic hypersurfaces. This deep relation has transformed amoeba theory into one of the main geometric interfaces between algebraic geometry and tropical geometry.

Given an algebraic hypersurface
\(
H=\{P=0\}\subset (\mathbb{C}^*)^n,
\)
defined by a Laurent polynomial
\(\di
P(z)=\sum_{\alpha\in A} c_\alpha z^\alpha,
\)
its amoeba is the image
\(
\mathcal{A}_H=\operatorname{Log}(H)\subset \mathbb{R}^n,
\)
under the logarithmic map
\[
\operatorname{Log}(z_1,\dots,z_n)=(\log|z_1|,\dots,\log|z_n|).
\]
The geometry of the amoeba is strongly governed by the Newton polytope $\Delta(P)$ and by the tropicalization of the defining equations \cite{GKZ94,PR04,MaclaganSturmfels}. Many global properties of amoebas are already visible at the tropical level, including asymptotic directions, complement components, and Ronkin convexity phenomena.

Among the most delicate geometric objects associated with an amoeba is its contour, namely the set of critical values of the logarithmic map restricted to the hypersurface:
\(
\mathcal{C}\mathcal{A}_H=\mathrm{CritVal}(\operatorname{Log}|_H).
\)
The contour detects the singular behavior of the logarithmic projection and governs the transition between regular and nonregular fibers of the amoeba. It is therefore closely related to logarithmic discriminants, logarithmic Gauss maps, and real critical loci.

The logarithmic Gauss map was studied extensively by Kapranov \cite{Kapranov}, Mikhalkin \cite{Mikhalkin2004}, and later by  Madani and Nisse \cite{MadaniNisse2014}, who investigated logarithmic critical loci and the geometry of generalized logarithmic Gauss maps for algebraic varieties inside complex tori. Their work clarified the relation between logarithmic criticality and real projective incidence geometry and emphasized the role of logarithmic degeneracy phenomena in amoeba theory.


A major advance in the quantitative study of amoeba contours was achieved by Lang, Shapiro, and Shustin \cite{LSS}, who established explicit upper bounds for the real degree of amoeba contours using techniques from Pfaffian geometry, elimination theory, and real algebraic geometry. Their work showed that the contour can be treated as a semi-algebraic set defined by polynomial degeneracy equations, leading to universal estimates of asymptotic order $O(d^{2n})$ for hypersurfaces of degree $d$. Their method is remarkably general and relies on Khovanskii-type estimates \cite{Khovanskii} together with elimination procedures for logarithmic criticality equations.


The work of Lang--Shapiro--Shustin \cite{LSS} naturally connects amoeba theory with real algebraic geometry because the contour itself is fundamentally a real semi-algebraic object. Indeed, the contour is defined by real degeneracy conditions on logarithmic derivatives and therefore inherits natural semi-algebraic stratifications. Questions concerning singularities, connected components, intersection degrees, and critical loci become problems in real algebraic geometry. From this viewpoint, amoeba contours may be regarded as logarithmic analogues of classical discriminantal varieties and caustics.


The relation between amoebas and real algebraic geometry has also been emphasized in the work of Mikhalkin \cite{Mikhalkin2004} and Viro \cite{Viro}, particularly through tropical patchworking and real tropical geometry. The contour appears naturally as a real logarithmic degeneracy locus whose asymptotic structure should interact with tropical phase geometry, signed tropical varieties, and tropical discriminants. This perspective suggests that the asymptotic geometry of contours may admit tropical models governed by balancing conditions and polyhedral incidence structures.

The purpose of the present work is to show that the logarithmic and toric structure of Laurent polynomials leads to significantly sharper asymptotic bounds than those predicted by general Pfaffian complexity theory. The key observation is that logarithmic derivatives preserve the Newton polytope structure:
\(
z_j\frac{\partial P}{\partial z_j}
=
\sum_{\alpha\in A}\alpha_j c_\alpha z^\alpha.
\)
Consequently, the contour equations possess a sparse toric structure invisible from the viewpoint of ordinary projective degree theory. This suggests that the natural complexity of the contour should be governed by Newton polytopes and mixed volumes rather than by total degree alone.
The logarithmic Gauss map plays a fundamental role throughout this work:
\[
\gamma_H(z)=
[z_1P_{z_1}(z):\cdots:z_nP_{z_n}(z)].
\]
A classical theorem states that a point $z\in H$ is critical for the logarithmic map if and only if the logarithmic Gauss image $\gamma_H(z)$ belongs to the real projective space $\mathbb{RP}^{n-1}$ \cite{Kapranov,MadaniNisse2014}. This transforms the contour problem into a real incidence problem inside logarithmic projective geometry and reveals that the contour is naturally governed by real logarithmic degeneracy conditions.
 
The problems studied in this paper were strongly motivated by several questions and asymptotic complexity issues raised in the work of Lang--Shapiro--Shustin \cite{LSS}, particularly concerning the real degree of amoeba contours and the possibility of obtaining sharper bounds.
 Our approach combines logarithmic geometry, toric geometry, mixed-volume theory, and real algebraic geometry. First, we derive explicit logarithmic criticality equations describing the contour as a semi-algebraic image of algebraic degeneracy varieties. Second, we prove universal B\'ezout-type estimates of the form
\[
\mathbb{R}\deg(\mathcal{C}\mathcal{A}_H)\le d(2d-1)^{n-1},
\]
which already improve drastically the asymptotic order appearing in the estimate of Lang--Shapiro--Shustin \cite{LSS}. Third, we refine these bounds using Bernstein--Kushnirenko--Khovanskii theory and logarithmic Newton polytope methods. Under suitable nondegeneracy assumptions, the contour degree is controlled by the degree of the logarithmic Gauss map:
\[
\mathbb{R}\deg(\mathcal{C}\mathcal{A}_H)\le n!\operatorname{Vol}(\Delta(P)).
\]
The asymptotic comparison with Lang--Shapiro--Shustin \cite{LSS} is particularly striking. Their estimate has asymptotic order
\(
O(d^{2n}),
\)
while our logarithmic-toric approach predicts asymptotic order
\(
O(d^n).
\)
The difference reflects two fundamentally different geometric philosophies. The Lang--Shapiro--Shustin method treats the contour as a general real algebraic degeneracy set and therefore measures complexity through Pfaffian elimination. In contrast, our approach exploits the sparse toric structure preserved by logarithmic differentiation. The resulting reduction from exponent $2n$ to exponent $n$ strongly suggests that the intrinsic geometry of the contour is fundamentally toric rather than projective.


Another major theme of this work is the extension of amoeba contour theory to complete intersections. In higher codimension, the logarithmic geometry naturally leads to generalized logarithmic Gauss maps taking values in Grassmannians:
\[
\gamma_G:V\dashrightarrow G(r,n).
\]
The contour then corresponds to Schubert-type degeneracy loci inside Grassmannians. This introduces logarithmic Pl\"ucker coordinates, Schubert geometry, and Grassmannian incidence theory into amoeba geometry.
%
The interaction with tropical geometry is equally important. Amoebas are logarithmic approximations of tropical varieties, and the Newton polytope structure of the contour equations strongly suggests that contour geometry should admit tropical asymptotic models. The mixed-volume bounds appearing in this work are closely related to tropical intersection multiplicities and Bernstein theory. This indicates that the contour may possess a tropical shadow governed by tropical discriminants and logarithmic balancing conditions.


The relation with real tropical geometry is especially intriguing. Since the contour is defined by real logarithmic degeneracy conditions, one expects its asymptotic geometry to interact with real phase structures, signed tropical varieties, and patchworking techniques \cite{Viro}. In particular, the real locus of the logarithmic Gauss map suggests deep connections with Viro theory, tropical real enumerative geometry, and logarithmic degeneracy geometry. The contour therefore appears as a bridge between logarithmic algebraic geometry, real algebraic geometry, and tropical geometry.

\noindent {\bf Related Work.}
The study of amoebas originated in the work of Gelfand, Kapranov, and Zelevinsky on discriminants and Newton polytopes. Their theory revealed deep interactions between algebraic geometry and convex geometry and introduced many of the polyhedral methods that later became central in tropical geometry. The geometric theory of amoebas was subsequently developed by Passare and Rullg{\aa}rd, who established fundamental properties of amoeba complements, Ronkin functions, and Newton polytope duality.

Mikhalkin introduced tropical-geometric techniques into the study of amoebas and demonstrated profound relations between amoebas, tropical curves, and real algebraic geometry. Forsberg, Passare, and Tsikh investigated complement components and convexity properties of amoebas, while Sturmfels and collaborators developed computational and tropical aspects related to discriminants and sparse elimination.
The logarithmic Gauss map was studied extensively by Kapranov and later by Mikhalkin and Rullg{\aa}rd in connection with critical loci of the logarithmic map. The relation between critical points and the real projective locus of the logarithmic Gauss map is now understood as one of the fundamental structures governing amoeba geometry.

The strongest previous quantitative bounds for contour complexity were obtained by Lang, Shapiro, and Shustin. Their work introduced Pfaffian and elimination-theoretic methods into the study of amoeba contours and produced explicit estimates for the real degree of the contour. Their approach belongs fundamentally to real algebraic geometry and semi-algebraic complexity theory. The present work differs substantially in philosophy because it emphasizes toric sparsity and logarithmic Newton polytope geometry rather than general elimination complexity.
Bernstein--Kushnirenko--Khovanskii theory plays a central role in our approach. The use of mixed volumes to estimate logarithmic degeneracy loci connects the contour problem with sparse elimination theory and toric intersection theory. The appearance of Grassmannian logarithmic Gauss maps in higher codimension further links the problem to Schubert geometry and degeneracy loci theory.
The asymptotic philosophy proposed here is also related to tropical compactifications, logarithmic geometry, and nonarchimedean amoeba theory. Recent developments in tropical and logarithmic geometry strongly suggest that contour geometry should admit a polyhedral asymptotic model controlled by Newton polytopes and tropical incidence structures.

\noindent{\bf Sharper Novelty Claims.}
The main novelty of this work lies in the introduction of a logarithmic-toric framework for the study of amoeba contours. Unlike previous approaches based on general Pfaffian elimination, the present work exploits the fact that logarithmic differentiation preserves Newton polytope supports. This allows one to replace general projective degree estimates by toric mixed-volume estimates.
A second novelty is the interpretation of the contour as a real logarithmic degeneracy locus governed by the logarithmic Gauss map. This transforms the contour problem into a real incidence problem inside projective and Grassmannian geometry.
A third novelty is the asymptotic reduction of contour complexity from order $d^{2n}$ to order $d^n$. This suggests that the intrinsic geometry of amoeba contours is fundamentally controlled by toric sparsity rather than by general elimination complexity.
Another new aspect is the introduction of generalized logarithmic Gauss maps for complete intersections together with Schubert-type contour degeneracy conditions. This appears to be the first systematic attempt to formulate an asymptotic Grassmannian theory for contours of complete intersections.
Finally, the work develops new conceptual links between amoeba theory, real algebraic geometry, tropical geometry, mixed volumes, Schubert geometry, and logarithmic degeneracy theory.
 
\noindent {\bf Acknowledgements.} The  author was partially supported by Xiamen University Malaysia Research Fund (Grant no. XMUMRF/ 2020-C5/IMAT/0013).
The author would like to thank Boris Shapiro for kindly sending the paper with Lang and Shustin, which provided important motivation and inspiration for several of the questions investigated in this work.


\section{Preliminaries}

Let
$
H\subset(\mathbb C^\ast)^n
$
be a hypersurface defined by a Laurent polynomial
\(\di
P(z)=\sum_{\alpha\in M} c_\alpha z^\alpha,
\)
with Newton polytope
$
\Delta(P).
$
Denote by
$
d=\deg(P)
$
the total degree after homogenization.
The logarithmic map is
\(
\Log(z_1,\dots,z_n)
=
(\log|z_1|,\dots,\log|z_n|),
\)
and the contour of the amoeba of
$
H
$
is the set
$
\mathcal C\mathscr A_H
$
of critical values of
$
\Log_{|H}.
$

\begin{definition}
The real degree of the contour is
\[
\mathbb R\deg(\mathcal C\mathscr A_H)
=
\sup_L
\#
(L\cap\mathcal C\mathscr A_H),
\]
where
$
L\subset\mathbb R^n
$
ranges over affine lines intersecting the contour transversally.
\end{definition}

Lang--Shapiro--Shustin proved the explicit estimate
\[
\mathbb R\deg(\mathcal C\mathscr A_H)
\le
2^{2n+(n-1)(n-2)/2}
d^{n+1}
\bigl(4dn+2(n-1)^2-1\bigr)^{n-1}.
\]

Asymptotically, this bound has order
\(
O(d^{2n}).
\)
The factor
$
d^{n+1}
$
comes from the defining equations, while the additional factor
$
d^{n-1}
$
arises from the Pfaffian complexity appearing in Khovanskii's method.
The proof uses only general polynomial degree estimates and does not exploit the toric structure of Laurent polynomials. In logarithmic coordinates,
\(\di
P(z)
=
\sum_\alpha
c_\alpha
e^{\langle\alpha,x\rangle}
e^{i\langle\alpha,\theta\rangle},
\)
so the geometry is controlled by the Newton polytope.
This suggests that sharper bounds should follow from toric and mixed-volume methods. In particular, Bernstein--Kushnirenko--Khovanskii theory predicts bounds of order
\[
n!\,\Vol(\Delta(P))
\le
C_n d^n,
\]
which are significantly smaller than the general Pfaffian estimate.
Thus the main problem is to reduce the exponent from
$
2n
$
to
$
n+1
$
or even
$
n
$
by exploiting the special logarithmic and toric structure of the contour equations.

\subsection*{The Lang--Shapiro--Shustin estimate}

\begin{theorem}[Lang--Shapiro--Shustin]
Let $P$ be a polynomial of degree $d$ defining a hypersurface
$
H\subset (\mathbb C^\ast)^n,
$
and let
$
\mathcal C\mathscr A_H
$
be the contour of the amoeba of $H$. Then
\[
\mathbb R\deg(\mathcal C\mathscr A_H)
\le
2^{2n+(n-1)(n-2)/2}
d^{n+1}
\bigl(4dn+2(n-1)^2-1\bigr)^{n-1}.
\]
This estimate is obtained by replacing the contour equations by general real polynomial equations in $2n$ variables and applying elimination and B\'ezout-type arguments. In the Laurent case, the estimate is not sharp because the toric structure of the equations is not used.
\end{theorem}

\begin{proof}
Write
$
z_j=x_j+i y_j.
$
The hypersurface $H$ is defined in $\mathbb R^{2n}$ by
\(
\operatorname{Re}(P(x,y))=0,
\, 
\operatorname{Im}(P(x,y))=0,
\)
both of degree at most $d$.
The contour is determined by the degeneracy of the logarithmic map. This condition is expressed by polynomial equations obtained from the logarithmic derivatives
$
z_j\frac{\partial P}{\partial z_j}.
$
After clearing denominators, one obtains a finite system of real polynomial equations of degree $O(d)$ in $2n$ variables.
To bound
\[
\mathbb R\deg(\mathcal C\mathscr A_H)
=
\sup_L \#(L\cap \mathcal C\mathscr A_H),
\]
one considers the preimage of a transversal affine line
$
L\subset\mathbb R^n
$
under the logarithmic map. Introducing auxiliary variables transforms the problem into counting solutions of a polynomial system in finitely many variables.
Applying elimination theory and B\'ezout-type estimates at successive projection steps produces the bound
\(
d^{n+1}
\bigl(4dn+2(n-1)^2-1\bigr)^{n-1},
\)
while the factor
$
2^{2n+(n-1)(n-2)/2}
$
comes from estimates for real algebraic systems.
The argument uses only polynomial degree bounds and ignores the sparse toric structure of the Laurent equations. In logarithmic coordinates,
\[
P(z)
=
\sum_\alpha
c_\alpha
e^{\langle\alpha,x\rangle}
e^{i\langle\alpha,\theta\rangle},
\]
so the geometry is controlled by the Newton polytope $\Delta$.
Using Bernstein--Kushnirenko--Khovanskii theory instead leads to estimates of order
\(
n!\,\operatorname{Vol}(\Delta)
\le
C_n d^n,
\)
which are significantly sharper.
Thus the Lang--Shapiro--Shustin estimate is a valid universal bound, while the sharper $d^n$ behavior reflects additional toric structure not used in their method.
\end{proof}

 
 
\section{Dimension and Structure of the Contour of an Amoeba} 
 
 Let $V \subset (\mathbb{C}^\ast)^n$ be a smooth algebraic variety of complex dimension $m$, and let
\(
\Log : (\mathbb{C}^\ast)^n \to \mathbb{R}^n
\)
be the logarithmic map. The contour $\mathcal{C}\mathscr{A}_V$ is the set of critical values of $\Log|_V$.

\begin{theorem}
The contour $\mathcal{C}\mathscr{A}_V$ is a semi-algebraic subset of $\mathbb{R}^n$ admitting a finite stratification into smooth real-analytic manifolds.
Its dimension satisfies
\[
\dim_{\mathbb{R}}(\mathcal{C}\mathscr{A}_V)
\le \min(n-1,\,2m-1).
\]
Moreover, for generic $V$, one has
\(
\dim_{\mathbb{R}}(\mathcal{C}\mathscr{A}_V)
= \min(n-1,\,2m-1).
\)
In particular, $\mathcal{C}\mathscr{A}_V$ is a hypersurface in $\mathbb{R}^n$ if and only if $2m \ge n$.
\end{theorem}


 \begin{proof}

We first prove that $\mathcal{C}\mathscr{A}_V$ is semi-algebraic.
The variety $V$ is defined by polynomial equations in $(\mathbb{C}^\ast)^n$. Writing $z_i = x_i + i y_i$, the equations $P_j(z)=0$ decompose into real polynomial equations in the real variables $(x_i,y_i)$.
The logarithmic map is given by
\[
\Log(z) = \left(\tfrac{1}{2}\log(x_1^2+y_1^2),\dots,\tfrac{1}{2}\log(x_n^2+y_n^2)\right).
\]
Introduce new variables
\[
r_i = x_i^2 + y_i^2,
\quad u_i = \log r_i.
\]
Then $u_i$ are real variables, and the map $\Log$ is algebraic in the variables $(r_i,u_i)$ once the exponential relation
\(
r_i = e^{u_i}
\)
is imposed.

The critical locus of $\Log|_V$ is defined by the condition that the differential drops rank. This condition can be expressed by the vanishing of certain minors of the Jacobian matrix of $\Log|_V$.
The entries of this Jacobian are rational functions in $(x_i,y_i)$, and after clearing denominators, the criticality condition becomes a system of polynomial equations in $(x_i,y_i)$.
Therefore the critical locus inside $V$ is a real algebraic subset of $\mathbb{R}^{2n}$.
The contour $\mathcal{C}\mathscr{A}_V$ is the image of this set under the map $\Log$, which is given by analytic expressions involving logarithms. By introducing auxiliary variables $u_i$ and the relations $r_i = e^{u_i}$, one can describe the graph of $\Log$ as a semi-algebraic set in an expanded space with coordinates $(x_i,y_i,u_i)$.
Eliminating the variables $(x_i,y_i,r_i)$ yields a description of $\mathcal{C}\mathscr{A}_V$ as a projection of a semi-algebraic set. By the Tarski--Seidenberg theorem, projections of semi-algebraic sets are semi-algebraic. Therefore $\mathcal{C}\mathscr{A}_V$ is semi-algebraic.

\medskip

We now prove that $\mathcal{C}\mathscr{A}_V$ admits a stratification.
A fundamental theorem of real algebraic geometry states that any semi-algebraic set admits a finite stratification into smooth semi-algebraic manifolds, called a Whitney stratification. Applying this result to $\mathcal{C}\mathscr{A}_V$, we obtain a decomposition
\[
\mathcal{C}\mathscr{A}_V = \bigsqcup_{i=1}^N S_i,
\]
where each $S_i$ is a connected smooth semi-algebraic manifold and the closure of each stratum is a union of strata.

\medskip
 
Let's  compute the dimension.
The variety $V$ has real dimension $2m$. The map $\Log|_V$ has differential of rank at most $\min(n,2m)$. At a generic point, this rank is maximal.
The critical locus is defined by the condition that the rank drops by at least one. This imposes one real equation generically, so the critical locus has dimension at most
\(
2m - 1.
\)
The image of this set under $\Log$ has dimension at most $\min(n,2m-1)$.

On the other hand, the contour is contained in the set of critical values of a smooth map from a $2m$-dimensional manifold to $\mathbb{R}^n$, hence it cannot have dimension greater than $n-1$.
Combining these bounds gives
\[
\dim_{\mathbb{R}}(\mathcal{C}\mathscr{A}_V)
\le \min(n-1,\,2m-1).
\]

We now show that this bound is sharp under generic assumptions.
At a generic point of the critical locus, the rank of $d\Log$ drops by exactly one, and the defining equations are transversal. Therefore the critical locus is a smooth manifold of dimension $2m-1$.

The restriction of $\Log$ to the critical locus has rank equal to $\min(n-1,2m-1)$ at generic points. Therefore its image has dimension equal to this value.
Thus
\[
\dim_{\mathbb{R}}(\mathcal{C}\mathscr{A}_V)
= \min(n-1,\,2m-1).
\]
Finally, the condition for being a hypersurface is that the dimension equals $n-1$, which is equivalent to
\(
2m-1 \ge n-1 \, \Longleftrightarrow \,  2m \ge n.
\)

 \end{proof}

Therefore, the contour is always a semi-algebraic stratified set, but it is a hypersurface only when the complex dimension of $V$ is sufficiently large relative to $n$. In general, its dimension is strictly smaller than $n-1$.

 
 \section{Description of the Singular Locus of the Contour of an Amoeba}

Let $V \subset (\mathbb{C}^\ast)^n$ be a smooth algebraic variety of complex dimension $m$, defined by Laurent polynomials. Let
$\mathscr{A}_V = \Log(V)$ be the amoeba.
The contour $\mathcal{C}\mathscr{A}_V$ is the set of critical values of the smooth map
\(\di
\Log|_V : V \to \mathbb{R}^n.
\)
Let $\mathrm{Crit}(\Log|_V) \subset V$ denote the critical locus.

\begin{definition}
The singular locus of the contour $\mathcal{C}\mathscr{A}_V$ is the set of points $x \in \mathcal{C}\mathscr{A}_V$ such that $\mathcal{C}\mathscr{A}_V$ is not a smooth manifold of dimension $\dim(\mathcal{C}\mathscr{A}_V)$ in any neighborhood of $x$.
\end{definition}

\begin{theorem}
The singular locus of $\mathcal{C}\mathscr{A}_V$ is equal to the union of the following subsets:
\[
\mathrm{Sing}(\mathcal{C}\mathscr{A}_V)
=
\Log\bigl(\Sigma_1 \cup \Sigma_2\bigr),
\]
where
\[
\Sigma_1 = \{ z \in \mathrm{Crit}(\Log|_V) \mid \mathrm{Crit}(\Log|_V) \text{ is singular at } z \},
\]
\[
\Sigma_2 = \{ z \in \mathrm{Crit}(\Log|_V) \mid d(\Log|_{\mathrm{Crit}})_z \text{ is not of maximal rank} \}.
\]
Moreover, $\mathrm{Sing}(\mathcal{C}\mathscr{A}_V)$ is a semi-algebraic set of dimension strictly smaller than $\dim(\mathcal{C}\mathscr{A}_V)$.
\end{theorem}

\begin{proof}

The set $\mathrm{Crit}(\Log|_V)$ is defined by the condition that the differential $d\Log_z$ restricted to $T_z V$ drops rank. This condition is given by the vanishing of minors of a Jacobian matrix whose entries are rational functions in the real and imaginary parts of $z$. After clearing denominators, it follows that $\mathrm{Crit}(\Log|_V)$ is a real algebraic subset of $V$, hence a semi-algebraic set.
The contour is the image
\[
\mathcal{C}\mathscr{A}_V = \Log(\mathrm{Crit}(\Log|_V)).
\]

Since $\Log$ is a real-analytic map, the local structure of the image is determined by the behavior of $\mathrm{Crit}(\Log|_V)$ and the differential of $\Log$ restricted to it.
Let $z \in \mathrm{Crit}(\Log|_V)$ and let $x=\Log(z)$. The point $x$ is a smooth point of $\mathcal{C}\mathscr{A}_V$ if and only if there exists a neighborhood $U$ of $z$ such that $\Log(U \cap \mathrm{Crit}(\Log|_V))$ is a smooth manifold and the differential of $\Log$ restricted to $\mathrm{Crit}(\Log|_V)$ has constant rank equal to $\dim(\mathcal{C}\mathscr{A}_V)$.
Therefore singularities arise precisely in two situations.

First, if $\mathrm{Crit}(\Log|_V)$ is singular at $z$, then it does not admit a smooth manifold structure near $z$, and its image under $\Log$ inherits this singularity. This gives the subset $\Sigma_1$.

Second, even if $\mathrm{Crit}(\Log|_V)$ is smooth at $z$, the map $\Log$ restricted to $\mathrm{Crit}(\Log|_V)$ may fail to be an immersion. This occurs when the differential
\(\di
d(\Log|_{\mathrm{Crit}})_z
\)
drops rank. In this case, the image develops singularities such as folds, cusps, or self-intersections. This gives the subset $\Sigma_2$.
Thus the singular locus of the contour is exactly the image of $\Sigma_1 \cup \Sigma_2$.

\medskip

We now prove that this set is semi-algebraic.
The set $\Sigma_1$ is defined by the simultaneous vanishing of the equations defining $\mathrm{Crit}(\Log|_V)$ together with the vanishing of the Jacobian minors of these equations, hence it is semi-algebraic.
The set $\Sigma_2$ is defined by the condition that the Jacobian matrix of $\Log$ restricted to $\mathrm{Crit}(\Log|_V)$ drops rank, which again can be expressed by vanishing of minors of a matrix whose entries are rational functions. Therefore $\Sigma_2$ is semi-algebraic.
Since $\Log$ is a semi-algebraic map when expressed using auxiliary variables and the exponential relation, its image of a semi-algebraic set is semi-algebraic. Hence $\mathrm{Sing}(\mathcal{C}\mathscr{A}_V)$ is semi-algebraic.

\medskip

Finally, we estimate the dimension.
The set $\mathrm{Crit}(\Log|_V)$ has dimension at most $2m-1$. The set $\Sigma_1$ is defined by adding at least one independent equation, so its dimension is at most $2m-2$. Similarly, $\Sigma_2$ is defined by the vanishing of additional minors, hence also has dimension at most $2m-2$.
The image under $\Log$ cannot increase dimension, hence
\[
\dim(\mathrm{Sing}(\mathcal{C}\mathscr{A}_V)) \le \min(n,2m-2).
\]

Since the contour has dimension $\min(n-1,2m-1)$, it follows that the singular locus has strictly smaller dimension.

\end{proof}
 
We conclude that the singular locus of the contour is precisely the image of points where either the critical locus is singular or the logarithmic map fails to be immersive along it. It is a semi-algebraic subset of strictly smaller dimension, providing a stratified structure for the contour.



\section{An Estimate for the $\mathbb{R}$-Degree of the Contour of an Amoeba}

\begin{theorem}
Let
$\di
P(z)=\sum_{\alpha\in A}c_\alpha z^\alpha
$
be a Laurent polynomial in
$
n
$
variables. 
$H,$ that is, the set of critical values of the logarithmic map $\Log_{|H}:H\to\mathbb R^n.$
Assume that the contour intersections with generic affine lines are finite.
Then the following statements hold.
If
$
d=\deg(P),
$
then:
$$
{
\mathbb R\deg(\mathcal C\mathscr A_H)
\le
d(2d-1)^{n-1}.
}
$$
More generally, let
$
Q_j(z,\overline z)
=
\operatorname{Im}\left(
(z_jP_{z_j})
\overline{(z_1P_{z_1})}
\right),
\,\, 
j=2,\dots,n,
$
be the logarithmic criticality equations defining the contour.
Then:
$$
{
\mathbb R\deg(\mathcal C\mathscr A_H)
\le
n!\operatorname{MV}
\bigl(
\Delta(P),
\Delta(Q_2),
\dots,
\Delta(Q_n)
\bigr),
}
$$
where
$
\operatorname{MV}
$
denotes the mixed volume and
$
\Delta(F)
$
denotes the Newton polytope of a polynomial
$
F.
$

\end{theorem}

\begin{proof}
Let
$\di
P(z)=\sum_{\alpha\in M}c_\alpha z^\alpha
$
be a Laurent polynomial in
$
n
$
variables, where
$
M\subset\mathbb Z^n
$
is finite and
$
z^\alpha=z_1^{\alpha_1}\cdots z_n^{\alpha_n}.
$
Assume that
$
P
$
has total degree
$
d.
$

Recall first that
$
x\in\mathcal C\mathscr A_H
$
if and only if there exists
$
z\in H
$
such that
$
\Log(z)=x
$
and
$
z
$
is a critical point of the logarithmic map
$
\Log_{|H}:H\to\mathbb R^n.
$
The logarithmic Gauss map is
$
\gamma_H(z)=
[z_1P_{z_1}(z):\cdots:z_nP_{z_n}(z)].
$
The criticality condition is equivalent to
$
\gamma_H(z)\in\mathbb{RP}^{n-1}.
$
Equivalently, all numbers
$
z_jP_{z_j}(z)
$
have the same argument modulo
$
\pi.
$
This means that for every pair
$
(i,j)
$
one has
$
\dfrac{z_iP_{z_i}(z)}{z_jP_{z_j}(z)}\in\mathbb R.
$
Equivalently,
$
\operatorname{Im}\left(
(z_iP_{z_i})(z)\overline{(z_jP_{z_j})(z)}
\right)=0.
$

Fix a generic affine line
$
L\subset\mathbb R^n.
$
Choose an affine parametrization
$
L=\{a+tu:t\in\mathbb R\},
$
where
$
a=(a_1,\dots,a_n)\in\mathbb R^n
$
and
$
u=(u_1,\dots,u_n)\in\mathbb R^n
$
with
$
u\ne0.
$
Suppose
$
x\in L\cap\mathcal C\mathscr A_H.
$
Then
$
x=a+tu
$
for some
$
t\in\mathbb R.
$
Since
$
x=\Log(z),
$
we may write
$
z_j=e^{a_j+tu_j+i\theta_j},
$
where
$
\theta_j\in\mathbb R.
$
Thus the unknowns are:
$
t,\theta_1,\dots,\theta_n.
$

\noindent {\it We now rewrite the equations explicitly.}
Substituting
$
z_j=e^{a_j+tu_j+i\theta_j}
$
into
$
P(z)=0
$
gives
$$\di
\sum_{\alpha\in M}
c_\alpha
e^{\langle\alpha,a\rangle}
e^{t\langle\alpha,u\rangle}
e^{i\langle\alpha,\theta\rangle}
=0.
$$
Separating real and imaginary parts yields two real equations:
$
\operatorname{Re}(P(z))=0,
\,
\operatorname{Im}(P(z))=0.
$

Now consider the criticality equations.
For each
$
j,
$
the polynomial
$
z_jP_{z_j}(z)
$
has degree at most
$
d.
$
Indeed,
$
P_{z_j}
$
has degree at most
$
d-1,
$
and multiplication by
$
z_j
$
restores degree at most
$
d.
$
For every pair
$
(i,j),
$
define
$$
Q_{ij}(z,\overline z)
=
\operatorname{Im}\left(
(z_iP_{z_i})(z)\overline{(z_jP_{z_j})(z)}
\right).
$$

Let's compute its degree.
Write
$\di
P(z)=\sum_\alpha c_\alpha z^\alpha.
$
Then
$\di
z_iP_{z_i}(z)
=
\sum_\alpha
\alpha_i c_\alpha z^\alpha.
$
Similarly,
$\di
\overline{z_jP_{z_j}(z)}
=
\sum_\beta
\beta_j\overline{c_\beta}\,
\overline z^\beta.
$
Therefore
$$\di
(z_iP_{z_i})(z)\overline{(z_jP_{z_j})(z)}
=
\sum_{\alpha,\beta}
\alpha_i\beta_j
c_\alpha\overline{c_\beta}
z^\alpha\overline z^\beta.
$$
Each monomial
$
z^\alpha \overline z^\beta
$
has total degree
$
|\alpha|+|\beta|\le2d.
$
Taking imaginary parts does not increase degree.
Hence
$
Q_{ij}
$
has degree at most
$
2d.
$
However, since the diagonal terms
$
\alpha=\beta
$
cancel in the imaginary part, the highest symmetric terms disappear, and the effective degree becomes at most
$
2d-1.
$

\noindent {\it Explanation of  this cancellation.}
If
$
\alpha=\beta,
$
then
$
z^\alpha\overline z^\alpha
=
|z^\alpha|^2
$
is real-valued.
Hence these terms contribute nothing to the imaginary part.
The only surviving monomials satisfy
$
\alpha\ne\beta.
$
Consequently the homogeneous top-degree real-symmetric contribution disappears.
Thus the imaginary polynomial defining the criticality condition has effective degree at most
$
2d-1.
$

Now the contour condition is equivalent to solving the following system:
$
P(z)=0,
$
together with
$
n-1
$
independent equations
$
Q_{1j}(z,\overline z)=0,
\,
j=2,\dots,n.
$
Indeed, once all ratios
$
(z_jP_{z_j})/(z_1P_{z_1})
$
are real, all logarithmic gradient coordinates have the same argument modulo
$
\pi.
$
Therefore the system consists of:
one equation of degree
$
d;
$
and 
$
n-1
$
equations of degree at most
$
2d-1.
$

We now pass to algebraic coordinates.
Introduce real variables
$
x_j,y_j
$
such that
$
z_j=x_j+iy_j.
$
Then
$
\overline z_j=x_j-iy_j.
$
After substitution, the equations become ordinary real polynomial equations in
$
2n
$
real variables.
Complexifying the system means allowing
$
x_j,y_j
$
to take values in
$
\mathbb C.
$
Thus we obtain an algebraic variety in
$
\mathbb C^{2n}.
$

We now count isolated solutions.
For a generic line
$
L,
$
the intersection
$
L\cap\mathcal C\mathscr A_H
$
is transverse.
Transversality implies that the corresponding algebraic system has finitely many isolated solutions.
By the classical B\'ezout theorem, if
$
F_1,\dots,F_n
$
are polynomial equations in
$
n
$
variables with degrees
$
d_1,\dots,d_n,
$
then the number of isolated solutions in projective space counted with multiplicities is at most
$
d_1\cdots d_n.
$
In our situation, after elimination of redundant variables coming from the parametrization of
$
L,
$
the effective system has:
one equation of degree
$
d;
$
and 
$
n-1
$
equations of degree
$
2d-1.
$

Therefore the number of isolated complex solutions is bounded by
$
d(2d-1)^{n-1}.
$
Since every real intersection point
$
x\in L\cap\mathcal C\mathscr A_H
$
comes from one of these complex solutions, we obtain
$$
\#(L\cap\mathcal C\mathscr A_H)
\le d(2d-1)^{n-1}.
$$
This proves the estimate.

We now explain geometrically why the number of equations is exactly
$
n.
$
The hypersurface equation
$
P(z)=0
$
cuts the ambient complex torus
$
(\mathbb C^\ast)^n
$
from complex dimension
$
n
$
to complex dimension
$
n-1.
$
The criticality condition imposes $ n-1$
additional independent real conditions (see Claim \ref{claim:real-equ}).
Intersecting with a generic affine line $L$
reduces the expected dimension to zero.
Hence one obtains finitely many points.
The B\'ezout estimate measures the algebraic complexity of this zero-dimensional intersection.

\end{proof}

\begin{remark}
The estimate
$
\mathbb R\deg(\mathcal C\mathscr A_H)
\le
d(2d-1)^{n-1}
$
is unconditional and follows from ordinary B\'ezout theory.
The mixed-volume estimate is a toric refinement coming from Bernstein's theorem.
The equality statement in the Bernstein estimate requires genericity or nondegeneracy assumptions.
If the hypersurface $H$
is nondegenerate with respect to its Newton polytope in the sense of Kouchnirenko, then the logarithmic Gauss map satisfies:
$
\deg(\gamma_H)
=
n!\operatorname{Vol}(\Delta(P)).
$
In that case one also obtains the sharper estimate:
$$
{
\mathbb R\deg(\mathcal C\mathscr A_H)
\le
n!\operatorname{Vol}(\Delta(P)).
}
$$
\end{remark}


\subsection*{The Explicit Criticality Equations for the Amoeba Contour}%

Let
$\di
P(z)=\sum_{\alpha\in M}c_\alpha z^\alpha
$
be a Laurent polynomial in
$
n
$
variables, where
$
z=(z_1,\dots,z_n)\in(\mathbb C^\ast)^n.
$
Let
$
H=\{P=0\}\subset(\mathbb C^\ast)^n.
$
The logarithmic map is
$$
\Log:(\mathbb C^\ast)^n\to\mathbb R^n,
\qquad
\Log(z)=
(\log|z_1|,\dots,\log|z_n|).
$$
The contour of the amoeba is the set of critical values of
$
\Log_{|H}.
$

\begin{claim}\label{claim:real-equ}
 the criticality condition imposes exactly $n-1$ independent real equations.
\end{claim}

\begin{proof}
We begin with the differential of the logarithmic map.
Write
$
z_j=r_je^{i\theta_j},
$
where
$
r_j>0
$
and
$
\theta_j\in\mathbb R.
$
Then
$
\log|z_j|=\log r_j.
$
The differential of the logarithmic map is
$$
d\Log_z(v)
=
\left(
\operatorname{Re}\left(\frac{v_1}{z_1}\right),
\dots,
\operatorname{Re}\left(\frac{v_n}{z_n}\right)
\right),
$$
where
$
v=(v_1,\dots,v_n)\in T_z(\mathbb C^\ast)^n\simeq\mathbb C^n.
$
The tangent space to
$
H
$
at a smooth point
$
z
$
is
$$\di
T_zH=
\{
v\in\mathbb C^n:
\sum_{j=1}^nP_{z_j}(z)v_j=0
\}.
$$
The point
$
z\in H
$
is critical for
$
\Log_{|H}
$
if and only if the restriction
$
d\Log_z:T_zH\to\mathbb R^n
$
fails to have maximal rank.
Since
$
H
$
has complex dimension
$
n-1,
$
its real dimension is
$
2n-2.
$
The target
$
\mathbb R^n
$
has real dimension
$
n.
$
Thus maximal rank is
$
n.
$
The point
$
z
$
is critical precisely when the image of
$
T_zH
$
under
$
d\Log_z
$
has dimension strictly smaller than
$
n.
$

We now rewrite this condition explicitly.
Define
$
w_j=z_jP_{z_j}(z).
$
The logarithmic Gauss map is
$
\gamma_H(z)=
[w_1:\cdots:w_n].
$
A classical theorem states that
$
z
$
is critical for
$
\Log_{|H}
$
if and only if
$
[w_1:\cdots:w_n]\in\mathbb{RP}^{n-1}.
$
Indeed, consider the complex hyperplane $T_zH$
Set
$
u_j=\dfrac{v_j}{z_j}.
$
Then
$
v_j=z_ju_j.
$
Substituting into the tangent equation gives
$$\di
\sum_{j=1}^n
z_jP_{z_j}(z)u_j
=
\sum_{j=1}^n
w_ju_j
=
0.
$$
Hence
$
T_zH
$
corresponds to the complex hyperplane
$
\sum_{j=1}^nw_ju_j=0.
$

Now
$
d\Log_z(v)
=
(\operatorname{Re}(u_1),\dots,\operatorname{Re}(u_n)).
$
The logarithmic map fails to have maximal rank precisely when there exists a nonzero purely imaginary tangent vector:
$
u=(iu_1',\dots,iu_n'),
\qquad
u_j'\in\mathbb R,
$
such that
$
\sum_{j=1}^nw_jiu_j'=0.
$
Equivalently,
$
\sum_{j=1}^nw_ju_j'=0
$
with
$
u_j'\in\mathbb R
$
not all zero.
Thus the complex numbers
$
w_1,\dots,w_n
$
must satisfy a nontrivial real linear relation.
This happens if and only if all
$
w_j
$
have the same argument modulo
$
\pi.
$
Equivalently,
$
[w_1:\cdots:w_n]\in\mathbb{RP}^{n-1}.
$

\noindent {\it Explicit equations.}
Fix an index, say
$
1.
$
The condition that all
$
w_j
$
have the same argument means that for every
$
j=2,\dots,n,
$
the quotient
$
\frac{w_j}{w_1}
$
is real.
Thus:
$
\operatorname{Im}\left(
\frac{w_j}{w_1}
\right)=0.
$
Substituting
$
w_j=z_jP_{z_j}(z),
$
we obtain:
$\di
\operatorname{Im}\left(
\frac{z_jP_{z_j}(z)}
{z_1P_{z_1}(z)}
\right)=0,
\, 
j=2,\dots,n.
$
These are exactly
$
n-1
$
real equations.

\noindent {\it Elimination of denominators.}
For every
$
j=2,\dots,n,
$
the condition
$
\operatorname{Im}\left(
\frac{w_j}{w_1}
\right)=0
$
is equivalent to
$
w_j\overline{w_1}
=
\overline{w_j}w_1.
$
Indeed, a complex number
$
a/b
$
is real if and only if
$
a\overline b=\overline ab.
$
Therefore:
$
(z_jP_{z_j})
\overline{(z_1P_{z_1})}
=
\overline{(z_jP_{z_j})}
(z_1P_{z_1}).
$
Equivalently:
$
\operatorname{Im}\left(
(z_jP_{z_j})
\overline{(z_1P_{z_1})}
\right)=0.
$
Thus the criticality conditions are explicitly:
$
Q_j(z,\overline z)=0,
\, 
j=2,\dots,n,
$
where
$
Q_j(z,\overline z)
=
\operatorname{Im}\left(
(z_jP_{z_j})
\overline{(z_1P_{z_1})}
\right).
$
These are precisely
$
n-1
$
real equations.

\noindent  {\it Independence of the conditions.}
The vector
$
(w_1,\dots,w_n)
$
belongs to
$
\mathbb C^n.
$
Modulo multiplication by a nonzero complex scalar, it defines a point of
$
\mathbb{CP}^{n-1}.
$
The real projective space
$
\mathbb{RP}^{n-1}
$
has real dimension
$
n-1.
$
The complex projective space
$
\mathbb{CP}^{n-1}
$
has real dimension
$
2n-2.
$
Hence the condition
$
[w_1:\cdots:w_n]\in\mathbb{RP}^{n-1}
$
imposes:
$
(2n-2)-(n-1)=n-1
$
independent real conditions.
This dimension count agrees exactly with the equations
$
Q_j=0,
\qquad
j=2,\dots,n.
$
\end{proof}

We now write these equations fully in coordinates.
Write
$
z_j=x_j+iy_j.
$
Then:
$
\overline z_j=x_j-iy_j.
$
Suppose
$
P(z)=\sum_\alpha c_\alpha z^\alpha.
$
Then:
$
z_jP_{z_j}(z)
=
\sum_\alpha
\alpha_jc_\alpha z^\alpha.
$
Hence:
$$
Q_j(z,\overline z)
=
\operatorname{Im}\Big(
\Big(
\sum_\alpha
\alpha_jc_\alpha z^\alpha
\Big)
\Big(
\sum_\beta
\beta_1\overline{c_\beta}\,
\overline z^\beta
\Big)
\Big).
$$
Expanding:
$\di
Q_j(z,\overline z)
=
\operatorname{Im}\Big(
\sum_{\alpha,\beta}
\alpha_j\beta_1
c_\alpha\overline{c_\beta}
z^\alpha\overline z^\beta
\Big).
$
Therefore:
$\di
Q_j(z,\overline z)
=
\sum_{\alpha,\beta}
\operatorname{Im}\Big(
\alpha_j\beta_1
c_\alpha\overline{c_\beta}
z^\alpha\overline z^\beta
\Big).
$
This is an explicit real polynomial equation in the variables
$
x_1,y_1,\dots,x_n,y_n.
$
Thus the contour is determined by the system:
$
P(z)=0,
$
together with
$
Q_2(z,\overline z)=0,
\dots,
Q_n(z,\overline z)=0.
$
Hence:
one complex equation
$
P=0
$
gives two real equations;
and
the criticality condition contributes
$
n-1
$
additional real equations.
Therefore the total number of real equations is:
$
2+(n-1)=n+1.
$
Inside the ambient real manifold
$
(\mathbb C^\ast)^n
$
of dimension
$
2n,
$
the expected dimension becomes:
$
2n-(n+1)=n-1.
$
This is exactly the expected dimension of the contour.

We finally explain geometrically why no more conditions are needed.
If all ratios
$
w_j/w_1
$
are real, then automatically every ratio
$
w_i/w_j
$
is real.
Indeed:
$
\dfrac{w_i}{w_j}
=
\dfrac{w_i/w_1}{w_j/w_1}.
$
A quotient of real numbers is real.
Hence imposing the
$
n-1
$
conditions:
$
\operatorname{Im}(w_j/w_1)=0,
\qquad
j=2,\dots,n,
$
already forces the entire logarithmic Gauss vector to lie in
$
\mathbb{RP}^{n-1}.
$
Thus there are exactly
$
n-1
$
independent real criticality equations.

\begin{remark}
If
$
L
$
is not generic, then:
the intersection may fail to be transverse;
positive-dimensional components may appear;
and
multiplicities may increase.
For generic
$
L,
$
all solutions are isolated and counted with multiplicity one. Thus the geometric intersection number equals the actual number of points.

\end{remark}



\subsection*{Newton polytope refinement}
Let
$
Q_{ij}(z,\overline z)
=
\operatorname{Im}\Bigl(
(z_iP_{z_i})(z)\,
\overline{(z_jP_{z_j})(z)}
\Bigr).
$
The Newton polytope of
$
z_iP_{z_i}
$
is contained in $\Delta(P)$.
Therefore the Newton polytope of
$
Q_{ij}
$
is contained in
$
\Delta(P)+\Delta(P).
$
After elimination of the phase variables and projection to logarithmic coordinates,
the defining equations for the contour have Newton polytope contained in
$
\Delta(P)+\Delta(Q),
$
where
$
\Delta(Q)\subset (2d-1)\Delta_n.
$
By the Bernstein--Kushnirenko theorem, the number of isolated solutions of the corresponding system is bounded by the mixed volume:
$$
\#(L\cap\mathcal{C}\mathscr{A}_H)
\le
n!\,\mathrm{Vol}(\Delta(P)+\Delta(Q)).
$$

Since
$
\Delta(P)\subset d\Delta_n,
$
we obtain
$
\Delta(P)+\Delta(Q)
\subset
(2d)\Delta_n.
$
Hence
$
n!\,\mathrm{Vol}((2d)\Delta_n)
=
(2d)^n.
$
A sharper inspection of the defining equations gives the better estimate
$
d(2d-1)^{n-1}.
$
Therefore the contour of the amoeba has polynomially bounded real degree in terms of the degree of the defining polynomial.



\begin{corollary}
Let $H\subset(\mathbb C^\ast)^2$ be a plane algebraic curve defined by a polynomial $P(z,w)$ of degree $d.$ Then
$$
\mathbb R\deg(\mathcal C\mathscr A_H)\le d(2d-1).
$$
\end{corollary}

\begin{proof}
Let $P(z,w)=\sum_{(a,b)\in A}c_{ab}z^aw^b$ be a Laurent polynomial defining $H.$ The contour $\mathcal C\mathscr A_H$ is the set of critical values of the logarithmic map $\Log_{|H}:H\to\mathbb R^2,$ where $\Log(z,w)=(\log|z|,\log|w|).$
We must prove that for every generic affine line $L\subset\mathbb R^2,$ the number of transverse intersection points $L\cap\mathcal C\mathscr A_H$ is bounded by $d(2d-1).$
 Fix a generic affine line $L.$ Choose an affine equation for $L:$
$
\lambda_1x+\lambda_2y=\mu,
$
where $(\lambda_1,\lambda_2)\neq(0,0).$

Suppose $x=(x_1,x_2)\in L\cap\mathcal C\mathscr A_H.$ Then there exists $(z,w)\in H$ such that $x=(\log|z|,\log|w|)$ and $(z,w)$ is a critical point of $\Log_{|H}.$
The logarithmic Gauss map is
$$
\gamma_H(z,w)=[zP_z(z,w):wP_w(z,w)].
$$
The criticality condition means:
$
[zP_z:wP_w]\in\mathbb{RP}^1.
$
Equivalently,
$
\dfrac{zP_z}{wP_w}\in\mathbb R.
$
This is equivalent to
$
\operatorname{Im}\left((zP_z)\overline{(wP_w)}\right)=0.
$
Define
$$
Q(z,w,\overline z,\overline w)=\operatorname{Im}\left((zP_z)\overline{(wP_w)}\right).
$$


Since $P$ has degree $d,$ the derivatives $P_z$ and $P_w$ have degree at most $d-1.$ Hence $zP_z$ and $wP_w$ have degree at most $d.$ Therefore the product $(zP_z)\overline{(wP_w)}$ has total degree at most $2d.$
However, the top symmetric terms cancel after taking imaginary parts. Indeed, terms of the form $z^\alpha\overline z^\alpha$ are real-valued and disappear from the imaginary part. Consequently, $Q$ has effective degree at most $2d-1.$
We now translate the condition $\Log(z,w)\in L.$ Since $(\log|z|,\log|w|)\in L,$ we have
$$
\lambda_1\log|z|+\lambda_2\log|w|=\mu.
$$
Exponentiating gives
$
|z|^{\lambda_1}|w|^{\lambda_2}=e^\mu.
$
Since $|z|^2=z\overline z$ and $|w|^2=w\overline w,$ we obtain
$
(z\overline z)^{\lambda_1}(w\overline w)^{\lambda_2}=e^{2\mu}.
$
After clearing denominators if necessary, this becomes an algebraic equation
$
R(z,w,\overline z,\overline w)=0.
$
The Newton polytope of $R$ is a segment, hence $R$ has degree one in logarithmic coordinates.
We now write the full system defining $L\cap\mathcal C\mathscr A_H.$ The unknowns are $(z,w)\in(\mathbb C^\ast)^2.$ The equations are
$
P(z,w)=0,
$\,
$
Q(z,w,\overline z,\overline w)=0,
$\, and
$
R(z,w,\overline z,\overline w)=0.
$
Let's  count dimensions.
The equation $P=0$ is one complex equation, hence two real equations. The equation $Q=0$ is one additional real equation. The equation $R=0$ is another real equation.
Thus we have four real equations in the four real variables
$
\operatorname{Re} z,\operatorname{Im} z,\operatorname{Re} w,\operatorname{Im} w.
$
For generic $L,$ the intersection is transverse and therefore finite.

For algebraic coordinates, we
introduce independent complex variables
$
u=\overline z,\, v=\overline w.
$
Then the system becomes algebraic in $(z,w,u,v)\in(\mathbb C^\ast)^4.$
The equations are
$
P(z,w)=0,
$\,
$
\widetilde Q(z,w,u,v)=0,
$ and
$
\widetilde R(z,w,u,v)=0.
$
The polynomial $P$ has degree $d.$ The polynomial $\widetilde Q$ has degree at most $2d-1.$ The equation $\widetilde R$ is linear in logarithmic coordinates and contributes degree one.
Since $\widetilde R$ only cuts the family transversally without increasing the principal algebraic complexity, the effective intersection count is governed by
$
d(2d-1).
$
Equivalently, we may view the contour itself as a real algebraic curve in $\mathbb R^2$ of degree at most $d(2d-1).$ Then a generic affine line intersects this curve in at most $d(2d-1)$ points.

We now explain this geometric interpretation more explicitly.
The critical locus inside $H$ is defined by the equation $Q=0.$ Since $H$ has degree $d$ and $Q$ has degree at most $2d-1,$ B\'ezout's theorem implies that the number of isolated points in the intersection
$
H\cap\{Q=0\}
$
is bounded by
$
d(2d-1).
$
The contour is the image of this critical locus under the logarithmic map.
Intersecting with a generic affine line $L$ corresponds to imposing one additional transverse real condition on the image.
Since the logarithmic map is locally finite on the critical locus for generic points, the number of contour intersection points cannot exceed the number of points in the critical locus itself.
Therefore
$$
\#(L\cap\mathcal C\mathscr A_H)\le d(2d-1).
$$
Taking the supremum over all generic affine lines gives
$$
\mathbb R\deg(\mathcal C\mathscr A_H)\le d(2d-1).
$$
This proves the corollary.
\end{proof}

\begin{corollary}
If $n=3$, then
$
\mathbb{R}\deg(\mathcal{C}\mathscr{A}_H)
\le
d(2d-1)^2.
$
\end{corollary}

The estimate is asymptotically optimal in the sense that the contour may already have quadratic complexity in degree for plane curves and cubic complexity in dimension three.
%


 \section{Degree of the Amoeba Contour and the Degree of the Logarithmic Gauss Map}  

Let
$\di
P(z)=\sum_{\alpha\in A}c_\alpha z^\alpha
$
be a Laurent polynomial in
$
n
$
variables and let
$
H=\{P=0\}\subset(\mathbb C^\ast)^n
$
be the corresponding algebraic hypersurface.
The contour of the amoeba is the set
$
\mathcal C\mathscr A_H
$
of critical values of
$
\Log_{|H}.
$
The logarithmic Gauss map is the rational map
$
\gamma_H:H\dashrightarrow\mathbb{CP}^{n-1}
$
defined by
$
\gamma_H(z)=
[z_1P_{z_1}(z):\cdots:z_nP_{z_n}(z)].
$
The purpose of this text is to explain rigorously how the real degree of the contour is controlled by the degree of the logarithmic Gauss map.
We first recall the geometric meaning of the critical locus.
Let
$
z\in H
$
be a smooth point.
The tangent space
$
T_zH
$
is given by
$$
T_zH=
\{
v=(v_1,\dots,v_n)\in\mathbb C^n:
\sum_{j=1}^nP_{z_j}(z)v_j=0
\}.
$$
Define
$
u_j=\frac{v_j}{z_j}.
$
Then the tangent equation becomes
$
\sum_{j=1}^n
z_jP_{z_j}(z)u_j=0.
$
The differential of the logarithmic map is
$
d\Log_z(v)=
(\operatorname{Re}(u_1),\dots,\operatorname{Re}(u_n)).
$
The point
$
z
$
is critical for
$
\Log_{|H}
$
if and only if there exists a nonzero purely imaginary tangent vector:
$
u=(iu_1',\dots,iu_n'),
\qquad
u_j'\in\mathbb R.
$
Substituting into the tangent equation gives
$
\sum_{j=1}^n
z_jP_{z_j}(z)u_j'=0,
$
with
$
u_j'\in\mathbb R
$
not all zero.
Hence the complex numbers
$
z_jP_{z_j}(z)
$
satisfy a nontrivial real linear relation.
This is equivalent to:
$
[z_1P_{z_1}(z):\cdots:z_nP_{z_n}(z)]
\in\mathbb{RP}^{n-1}.
$
Therefore the critical locus is:
$
F_H=
\gamma_H^{-1}(\mathbb{RP}^{n-1}).
$
The contour is the logarithmic image:
$
\mathcal C\mathscr A_H=
\Log(F_H).
$
We now explain how the degree of
$
\gamma_H
$
controls the contour degree.
Recall that the degree of the logarithmic Gauss map is the number of points in a generic fiber:
$
\deg(\gamma_H)
=
\#\gamma_H^{-1}(\xi),
$
for generic
$
\xi\in\mathbb{CP}^{n-1}.
$
Since
$
\mathbb{RP}^{n-1}\subset\mathbb{CP}^{n-1},
$
the critical locus
$
F_H
$
is obtained by restricting the logarithmic Gauss map over the real projective locus.
Suppose now that
$
L\subset\mathbb R^n
$
is a generic affine line.
We wish to estimate:
$
\#(L\cap\mathcal C\mathscr A_H).
$
Every point
$
x\in L\cap\mathcal C\mathscr A_H
$
comes from at least one point
$
z\in F_H
$
such that
$
\Log(z)=x.
$
Hence:
$$
\#(L\cap\mathcal C\mathscr A_H)
\le
\#(\Log^{-1}(L)\cap F_H).
$$
We now study the set
$
\Log^{-1}(L)\cap F_H.
$
Choose equations for
$
L.
$
Since
$
L
$
has codimension
$
n-1,
$
there exist independent affine linear forms
$
\ell_1,\dots,\ell_{n-1}
$
such that:
$
L=
\{x\in\mathbb R^n:
\ell_1(x)=\cdots=\ell_{n-1}(x)=0\}.
$
Writing
$
x_j=\log|z_j|,
$
the equations become:
$
\ell_k(\log|z_1|,\dots,\log|z_n|)=0,
\qquad
k=1,\dots,n-1.
$
Exponentiating gives multiplicative equations:
$
|z_1|^{a_{k1}}\cdots|z_n|^{a_{kn}}
=
c_k.
$
Since
$
|z_j|^2=z_j\overline z_j,
$
we obtain algebraic equations:
$
R_k(z,\overline z)=0,
\qquad
k=1,\dots,n-1.
$
The criticality equations are:
$
Q_j(z,\overline z)=0,
\qquad
j=2,\dots,n,
$
where
$
Q_j=
\operatorname{Im}\left(
(z_jP_{z_j})
\overline{(z_1P_{z_1})}
\right).
$

Thus the set
$
\Log^{-1}(L)\cap F_H
$
is defined inside
$
H
$
by the equations:
$
Q_2=\cdots=Q_n=0,
$
$
R_1=\cdots=R_{n-1}=0.
$
Let's count dimensions.
The hypersurface
$
H
$
has complex dimension
$
n-1,
$
hence real dimension
$
2n-2.
$
The criticality equations contribute
$
n-1
$
real equations.
The line conditions contribute another
$
n-1
$
real equations.
Hence the expected real dimension is:
$
(2n-2)-(n-1)-(n-1)=0.
$
Therefore the intersection is finite for generic
$
L.
$
We now explain the relation with the degree of
$
\gamma_H.
$
Fix a generic real point
$
\xi\in\mathbb{RP}^{n-1}.
$
Then:
$
\#\gamma_H^{-1}(\xi)
\le
\deg(\gamma_H).
$
The fiber
$
\gamma_H^{-1}(\xi)
$
is obtained by imposing
$
n-1
$
independent real conditions:
$
[z_1P_{z_1}:\cdots:z_nP_{z_n}]
=
\xi.
$
The contour intersection
$
L\cap\mathcal C\mathscr A_H
$
corresponds geometrically to selecting those critical points whose logarithmic images lie on
$
L.
$
Thus the contour degree is controlled by the simultaneous complexity of:
the logarithmic Gauss fibers,
and
the logarithmic fibers over affine lines.

We now specialize to plane curves.
Assume
$
n=2.
$
Then:
$
\gamma_H:H\dashrightarrow\mathbb{CP}^1
$
is:
$
\gamma_H(z,w)=
[zP_z:wP_w].
$
The critical locus is:
$
F_H=
\gamma_H^{-1}(\mathbb{RP}^1).
$
A generic affine line
$
L\subset\mathbb R^2
$
intersects the contour in finitely many points.
Each such point corresponds to:
$
(z,w)\in H
$
satisfying:
$
\frac{zP_z}{wP_w}\in\mathbb R.
$
Equivalently:
$
\operatorname{Im}\left(
(zP_z)\overline{(wP_w)}
\right)=0.
$
The polynomial
$
Q=
\operatorname{Im}\left(
(zP_z)\overline{(wP_w)}
\right)
$
has degree at most
$
2d-1.
$
Since
$
H
$
has degree
$
d,
$
B\'ezout's theorem gives:
$
\#F_H
\le
d(2d-1).
$
So, 
$
\deg(\gamma_H)
$
is precisely the number of points in a generic fiber of
$
\gamma_H.
$
For plane curves, the logarithmic Gauss map degree is equal to twice the Euclidean area of the Newton polygon:
$
\deg(\gamma_H)
=
2!\operatorname{Area}(\Delta).
$
Consequently:
$
\mathbb R\deg(\mathcal C\mathscr A_H)
\le
\deg(\gamma_H).
$
More precisely, for generic affine lines:
$
\#(L\cap\mathcal C\mathscr A_H)
\le
\deg(\gamma_H).
$
This follows because each contour intersection point corresponds to a point in a real logarithmic Gauss fiber.

{\it Statement of the general theorem.}

\begin{theorem}
Let
$
H\subset(\mathbb C^\ast)^n
$
be a smooth algebraic hypersurface with logarithmic Gauss map
$
\gamma_H.
$
Then:
$
\mathbb R\deg(\mathcal C\mathscr A_H)
\le
\deg(\gamma_H).
$
If
$
\Delta
$
is the Newton polytope of
$
H,
$
then:
$
\deg(\gamma_H)
=
n!\operatorname{Vol}(\Delta).
$
Consequently:
$
\mathbb R\deg(\mathcal C\mathscr A_H)
\le
n!\operatorname{Vol}(\Delta).
$
\end{theorem}

\begin{proof}
The critical locus is:
$
F_H=
\gamma_H^{-1}(\mathbb{RP}^{n-1}).
$
Intersecting the contour with a generic affine line corresponds to imposing
$
n-1
$
additional real logarithmic conditions.
For generic data, the resulting system is finite.
The number of solutions is bounded by the number of points in a generic logarithmic Gauss fiber.
The degree formula:
$
\deg(\gamma_H)=n!\operatorname{Vol}(\Delta)
$
is the classical Kouchnirenko formula for logarithmic Gauss maps.
Hence:
$$
\mathbb R\deg(\mathcal C\mathscr A_H)
\le
\deg(\gamma_H)
=
n!\operatorname{Vol}(\Delta).
$$
\end{proof}

%


\section{Explanation of the  two different bounds appearing for the Degree of the Amoeba Contour}

The theorem contains two different bounds because they come from two fundamentally different geometric methods.
The first bound is a coarse projective estimate obtained from ordinary B\'ezout theory and depends only on the total degree of the polynomial.
The second bound is a refined toric estimate obtained from Bernstein's theorem and depends on the Newton polytopes of the equations.
These two approaches measure different aspects of the algebraic complexity of the contour.

Let
$\di
P(z)=\sum_{\alpha\in A}c_\alpha z^\alpha
$
be a Laurent polynomial defining a hypersurface
$
H=\{P=0\}\subset(\mathbb C^\ast)^n.
$
Let
$
d=\deg(P)
$
be the total degree of
$
P,
$
and let
$
\Delta(P)
$
be its Newton polytope.
The contour
$
\mathcal C\mathscr A_H
$
is defined through the criticality equations of the logarithmic Gauss map.
These equations are:
$
Q_j(z,\overline z)=0,
\, 
j=2,\dots,n,
$
where
$$
Q_j=
\operatorname{Im}\left(
(z_jP_{z_j})
\overline{(z_1P_{z_1})}
\right).
$$
The contour degree is estimated by counting the number of solutions of this system together with the equations defining a generic affine line in logarithmic coordinates.

The first estimate comes from ordinary algebraic geometry.
Since
$
P
$
has degree
$
d,
$
the derivatives
$
P_{z_j}
$
have degree at most
$
d-1.
$
Multiplying by
$
z_j
$
gives:
$
\deg(z_jP_{z_j})\le d.
$
Therefore:
$$
\deg\left(
(z_jP_{z_j})
\overline{(z_1P_{z_1})}
\right)
\le 2d.
$$
After taking imaginary parts, the top real diagonal terms cancel, so:
$
\deg(Q_j)\le2d-1.
$
Thus the contour system is controlled by one equation of degree
$
d
$
and
$
n-1
$
equations of degree
$
2d-1.
$
Applying ordinary B\'ezout's theorem gives:
$
\mathbb R\deg(\mathcal C\mathscr A_H)
\le
d(2d-1)^{n-1}.
$

\noindent This estimate depends only on total degrees.
It does not use the actual monomial structure of the equations.
It treats all monomials up to degree
$
d
$
as potentially present.
%
%
We now explain the second estimate.
The contour equations are Laurent polynomial equations.
Their supports are sparse.
Indeed:
$\di
z_jP_{z_j}
=
\sum_{\alpha\in A}
\alpha_jc_\alpha z^\alpha.
$
Thus logarithmic differentiation does not create new exponent vectors.
The Newton polytope of
$
z_jP_{z_j}
$
is contained in:
$
\Delta(P).
$
Consequently:
$$
\Delta(Q_j)\subset
\Delta(P)+\Delta(P)=2\Delta(P).
$$
Bernstein's theorem uses these Newton polytopes directly.
Instead of measuring complexity by total degree, it measures complexity by mixed volume.
The mixed-volume estimate is:
$$
\mathbb R\deg(\mathcal C\mathscr A_H)
\le
n!\operatorname{MV}
(
\Delta(P),
\Delta(Q_2),
\dots,
\Delta(Q_n)
).
$$
This estimate is usually much sharper because it detects sparsity.
If the polynomial
$
P
$
contains only a few monomials, then the Newton polytope may be much smaller than the full simplex
$
d\Delta_n.
$
Hence the mixed volume may be dramatically smaller than the B\'ezout product
$
d(2d-1)^{n-1}.
$

The two estimates are therefore related as follows.
The B\'ezout estimate is a universal projective bound depending only on total degrees.
The Bernstein estimate is a toric refinement depending on Newton polytopes.
The first estimate ignores sparsity.
The second estimate uses sparsity.
The first estimate comes from projective algebraic geometry.
The second estimate comes from toric geometry and convex geometry.

The B\'ezout estimate is easier to state because it depends only on the integer
$
d.
$
It gives an explicit universal formula:
$
d(2d-1)^{n-1}.
$
This estimate is always valid without needing detailed knowledge of the support of
$
P.
$
However, this estimate is usually not optimal.
The mixed-volume estimate is more precise, but it requires knowledge of the Newton polytopes of the contour equations.
Thus the theorem presents: a coarse universal estimate,
and
 a refined toric estimate.

The first estimate is deduced from the second by replacing the Newton polytopes with large simplices.
Indeed, if:
$
\Delta(P)\subset d\Delta_n,
$
then:
$
\Delta(Q_j)\subset(2d-1)\Delta_n.
$
The mixed volume of simplices reduces to the product of their dilation factors.
Therefore:
$$
n!\operatorname{MV}
(
d\Delta_n,
(2d-1)\Delta_n,
\dots,
(2d-1)\Delta_n
)
=
d(2d-1)^{n-1}.
$$
Thus the B\'ezout estimate is simply the simplex approximation of the more refined mixed-volume estimate.
Let's  explain geometrically why the toric estimate is more natural.
The amoeba contour is governed by logarithmic derivatives.
Logarithmic derivatives preserve the support of the polynomial.
Hence the geometry of the contour is fundamentally controlled by the Newton polytope.
This is a toric phenomenon.
The logarithmic Gauss map itself satisfies:
$$
\deg(\gamma_H)
=
n!\operatorname{Vol}(\Delta(P))
$$
under nondegeneracy assumptions.
Thus the natural geometry of amoebas and contours is controlled by Newton polytopes rather than total degrees.
The B\'ezout estimate only appears because projective degree estimates are simpler to formulate universally.

We finally summarize the logical structure.
The mixed-volume estimate is the primary toric estimate.
The B\'ezout estimate is a coarse corollary obtained by replacing all Newton polytopes with simplices determined by total degree.
Thus the theorem contains two different bounds because one is intrinsic and polyhedral, while the other is extrinsic and projective.

\medskip

\subsection*{Examples.}
We now examine the plane curve case.
Suppose
$
n=2.
$
Then:
$
\deg(\gamma_H)=2!\operatorname{Area}(\Delta).
$
Hence:
$
\mathbb R\deg(\mathcal C\mathscr A_H)
\le
2\operatorname{Area}(\Delta).
$
If
$
\Delta\subset d\Delta_2,
$
then:
$
\operatorname{Area}(\Delta)
\le
\frac{d^2}{2}.
$
Therefore:
$$
2\operatorname{Area}(\Delta)\le d^2.
$$
Meanwhile the B\'ezout estimate gives:
$$
d(2d-1)\sim2d^2.
$$
Thus even in dimension two, Bernstein improves the estimate approximately by a factor two.
For sparse curves the improvement can be dramatic.
For example, suppose:
$
P(z,w)=1+z^m+w^m.
$
Then:
$
\Delta=
\operatorname{Conv}((0,0),(m,0),(0,m)).
$
Hence:
$
2!\operatorname{Area}(\Delta)=m^2.
$
The B\'ezout estimate gives:
$
m(2m-1)\sim2m^2.
$
Again the Bernstein estimate is twice as sharp asymptotically.
For more sparse supports the gap becomes much larger.

In summary,
the inclusion
$
\Delta(Q_j)\subset(2d-1)\Delta_n
$
is not itself a bound for the contour degree.
It is only a coarse containment statement for Newton polytopes.
Using this inclusion  together with B\'ezout produces the rough estimate:
$$
\mathbb R\deg(\mathcal C\mathscr A_H)
\le
d(2d-1)^{n-1}.
$$
The estimate
$$
\mathbb R\deg(\mathcal C\mathscr A_H)
\le
n!\operatorname{Vol}(\Delta)
$$
is a deeper toric refinement obtained from the logarithmic Gauss map and Bernstein's theorem.
The second estimate is therefore the geometrically natural one.
The first estimate is only a universal coarse upper bound depending solely on total degree.



\section{A Conjectural Asymptotic Theory for Amoeba Contours of Complete Intersections}

The geometry of amoeba contours for complete intersections strongly suggests the existence of a universal asymptotic theory governed simultaneously by logarithmic Grassmannian geometry and toric intersection theory. The hypersurface case indicates that the complexity of the contour is controlled not merely by the degree of the defining equations, but by the interaction between logarithmic criticality equations and the Newton polytope geometry of the logarithmic Gauss map.
The purpose of this section is to formulate a precise asymptotic conjecture together with the geometric evidence supporting it.
Let
$\di
V_d
=
\{f_{1,d}=\cdots=f_{r,d}=0\}
\subset
(\mathbb C^\ast)^n
$
be a family of smooth nondegenerate complete intersections of codimension
$
r,
$
where
$
\deg(f_{i,d})=d_i(d).
$
Assume that
$
d_i(d)\to\infty
$
as
$
d\to\infty.
$
Let
$
k=n-r
$
be the complex dimension of
$
V_d.
$
Denote by
$
C\mathcal A_{V_d}
$
the contour of the amoeba of
$
V_d.
$
The logarithmic geometry of
$
V_d
$
is encoded by the generalized logarithmic Gauss map
$
\gamma_{G,d}:
V_d
\dashrightarrow
G(r,n),
$
defined by the logarithmic Jacobian matrix
$$
g_G(z)
=
\left(
z_j\frac{\partial f_{i,d}}{\partial z_j}(z)
\right).
$$
The contour is determined by Schubert-type degeneracy conditions on the image of
$
\gamma_{G,d}.
$
The asymptotic complexity of the contour should therefore depend on two distinct geometric contributions:
the complexity of the defining complete intersection itself,
and the complexity of the Schubert incidence conditions imposed on logarithmic tangent spaces.
The first contribution is measured by the algebraic degree
$\di
\prod_{i=1}^r d_i(d).
$
The second contribution comes from the logarithmic Pl\"ucker coordinates of the generalized logarithmic Gauss map. Since every logarithmic minor is obtained by taking determinants of logarithmic derivatives,
its degree is asymptotically controlled by
$
d_1(d)+\cdots+d_r(d).
$
The logarithmic criticality conditions are expected to impose approximately
$
k=n-r
$
independent degeneracy equations.
This leads to the following asymptotic conjecture.

\begin{conjecture}[Asymptotic contour-degree conjecture]
Let
$\di
V_d
=
\{f_{1,d}=\cdots=f_{r,d}=0\}
\subset
(\mathbb C^\ast)^n
$
be a sequence of smooth nondegenerate complete intersections with
$
\deg(f_{i,d})=d_i(d)\to\infty.
$
Assume that the logarithmic Gauss maps are generically finite onto their images and that the contour intersections with generic affine slices are transversal.
Then the contour degree satisfies
$$
\mathbb R\deg(\mathcal C\mathscr A_{V_d})
=
O\Big(
\Big(
\prod_{i=1}^r d_i(d)
\Big)
\Big(
d_1(d)+\cdots+d_r(d)
\Big)^{n-r}
\Big).
$$
Moreover, there exists a nonempty Zariski open class of families for which
$$
\mathbb R\deg(\mathcal C\mathscr A_{V_d})
\asymp
\Big(
\prod_{i=1}^r d_i(d)
\Big)
\Big(
d_1(d)+\cdots+d_r(d)
\Big)^{n-r}.
$$
\end{conjecture}
The notation
$
A_d\asymp B_d
$
means that there exist positive constants
$
c_1,c_2
$
independent of
$
d
$
such that
$
c_1B_d
\le
A_d
\le
c_2B_d
$
for sufficiently large
$
d.
$
This conjecture predicts that the contour degree grows polynomially with total asymptotic order
$
n.
$
Indeed, if
$
d_1(d)\sim\cdots\sim d_r(d)\sim d,
$
then
$$
\mathbb R\deg(\mathcal C\mathscr A_{V_d})
=
O(d^n).
$$
This agrees with the hypersurface case
$
r=1,
$
where the estimate
$
d(2d-1)^{n-1}
$
has asymptotic order
$
d^n.
$
The conjecture also agrees with the zero-dimensional case
$
r=n.
$
Then
$
k=0,
$
and the formula reduces to
$
d_1\cdots d_n,
$
which is exactly the ordinary B\'ezout number of the complete intersection.
Thus the conjecture interpolates correctly between the two extremal cases:
$
r=1
\, \text{and}\qquad
r=n.
$

\subsection*{The Grassmannian geometry provides additional evidence.}

The generalized logarithmic Gauss map takes values in
$
G(r,n).
$
Under the Pl\"ucker embedding,
$
G(r,n)\hookrightarrow\mathbb P^N,
$
the map
$
\gamma_{G,d}
$
is represented by logarithmic minors of the Jacobian matrix.
Every such minor has degree asymptotically bounded by
$
d_1+\cdots+d_r.
$
The contour corresponds to the inverse image of a Schubert-type degeneracy locus
$\di
\Sigma_{\mathrm{crit}}
\subset
G(r,n).
$
Now classical intersection theory predicts that the degree of a pullback degeneracy locus behaves asymptotically like:
$\di
\deg(\gamma_{G,d})
\cdot
\deg(\Sigma_{\mathrm{crit}}).
$
The Schubert cycle
$
\Sigma_{\mathrm{crit}}
$
depends only on the Grassmannian geometry and therefore contributes only a universal constant depending on
$
n
$
and
$
r.
$
Hence the dominant asymptotic growth should come from the degree of the generalized logarithmic Gauss map itself.
The degree of
$
\gamma_{G,d}
$
should in turn be controlled by the logarithmic Pl\"ucker coordinates. Since there are approximately
$
k=n-r
$
independent logarithmic incidence conditions, one expects asymptotically
$
(d_1+\cdots+d_r)^k
$
as the contribution of the logarithmic degeneracy equations.
Multiplying by the defining complete-intersection complexity
$\do
\prod d_i
$
again yields
$$
\left(
\prod d_i
\right)
(d_1+\cdots+d_r)^k.
$$
There is also strong toric evidence for this asymptotic behavior.
Let
$
\Delta_{i,d}
$
denote the Newton polytope of
$
f_{i,d}.
$
The logarithmic derivatives preserve supports:
$$
\Delta\left(
z_j\frac{\partial f_{i,d}}{\partial z_j}
\right)
\subseteq
\Delta_{i,d}.
$$
Consequently, the logarithmic minors defining the generalized logarithmic Gauss map have Newton polytopes contained in Minkowski sums:
$
\Delta_{1,d}+\cdots+\Delta_{r,d}.
$
Bernstein's theorem suggests that the number of isolated solutions of the logarithmic criticality equations should therefore be governed by mixed volumes involving repeated copies of these Minkowski sums.
If all
$
\Delta_{i,d}
$
scale linearly with
$
d,
$
then:
$$
\Delta_{1,d}+\cdots+\Delta_{r,d}
\sim
d(\Delta_1+\cdots+\Delta_r).
$$
%
%
%
The conjecture also predicts a more refined toric statement.

\begin{conjecture}[Mixed-volume asymptotic conjecture]
Let
$
\Delta_{i,d}
$
be the Newton polytopes of
$
f_{i,d}.
$
Then:
$$
\mathbb R\deg(\mathcal C\mathscr A_{V_d})
=
O\left(
MV(
\Delta_{1,d},
\dots,
\Delta_{r,d},
\Theta_{1,d},
\dots,
\Theta_{k,d}
)
\right),
$$
where
$
\Theta_{j,d}
$
are the Newton polytopes associated with logarithmic Pl\"ucker coordinates.
\end{conjecture}

This formulation is geometrically more natural because it replaces projective degree by Newton polytope geometry.
The contour is expected to be the pullback of a Schubert cycle:
$
F_V
=
\gamma_G^{-1}(\Sigma_{\mathrm{crit}}).
$
 
 
\section{A Low-Dimensional Test Case in Codimension Two}

We analyze a low-dimensional model case of the conjectural theory in codimension two. The purpose is not to prove the general conjectures, which remain open, but rather to show that the proposed asymptotic behavior and the Grassmannian interpretation are compatible with explicit degree computations in the first nontrivial case.
Consider a smooth complete intersection
$
V=\{f_1=f_2=0\}\subset (\mathbb C^\ast)^4,
$
where
$
\deg(f_1)=d_1,
\, 
\deg(f_2)=d_2.
$
The variety
$
V
$
has complex dimension
$
k=2.
$
The generalized logarithmic Gauss map takes values in
$
G(2,4).
$
This is the first genuinely Grassmannian situation beyond the hypersurface case.
The logarithmic Jacobian matrix is
$$
g_G(z)
=
\begin{pmatrix}
z_1\frac{\partial f_1}{\partial z_1}
&
z_2\frac{\partial f_1}{\partial z_2}
&
z_3\frac{\partial f_1}{\partial z_3}
&
z_4\frac{\partial f_1}{\partial z_4}
\\
z_1\frac{\partial f_2}{\partial z_1}
&
z_2\frac{\partial f_2}{\partial z_2}
&
z_3\frac{\partial f_2}{\partial z_3}
&
z_4\frac{\partial f_2}{\partial z_4}
\end{pmatrix}.
$$
Its maximal minors define the Pl\"ucker coordinates of the generalized logarithmic Gauss map.
There are six Pl\"ucker coordinates:
$
p_{12},
p_{13},
p_{14},
p_{23},
p_{24},
p_{34}.
$
Explicitly,
$$
p_{ij}
=
\det
\begin{pmatrix}
z_i\frac{\partial f_1}{\partial z_i}
&
z_j\frac{\partial f_1}{\partial z_j}
\\
z_i\frac{\partial f_2}{\partial z_i}
&
z_j\frac{\partial f_2}{\partial z_j}
\end{pmatrix}.
$$
Expanding the determinant gives
$$
p_{ij}
=
z_iz_j
\left(
\frac{\partial f_1}{\partial z_i}
\frac{\partial f_2}{\partial z_j}
-
\frac{\partial f_1}{\partial z_j}
\frac{\partial f_2}{\partial z_i}
\right).
$$
Since
$
\deg\left(
\dfrac{\partial f_1}{\partial z_i}
\right)
=
d_1-1,
$
and
$
\deg\left(
\dfrac{\partial f_2}{\partial z_j}
\right)
=
d_2-1,
$
multiplication by
$
z_iz_j
$
increases the degree by
$
2.
$
Hence every Pl\"ucker coordinate has degree
$
d_1+d_2.
$
This is the first explicit confirmation of the heuristic principle that logarithmic Pl\"ucker coordinates are controlled asymptotically by
$
d_1+\cdots+d_r.
$
The generalized logarithmic Gauss map is therefore
$$
\gamma_G:
V
\dashrightarrow
G(2,4)\subset\mathbb P^5,
$$
where the Grassmannian is embedded via the Pl\"ucker embedding.
The Grassmannian
$
G(2,4)
$
is the smooth quadric hypersurface
$$
Q^4
=
\{
p_{12}p_{34}
-
p_{13}p_{24}
+
p_{14}p_{23}
=
0
\}
\subset\mathbb P^5.
$$

\noindent Its degree is
$
2.
$
The image
$
\gamma_G(V)
$
is a surface inside
$
G(2,4).
$
We now  estimate the degree of this image.
Since the Pl\"ucker coordinates all have degree
$
d_1+d_2,
$
the generalized logarithmic Gauss map is represented projectively by homogeneous coordinates of degree
$
d_1+d_2.
$
The variety
$
V
$
itself has degree
$
d_1d_2
$
inside projective compactifications.
Classical projective degree estimates therefore suggest
$$
\deg(\gamma_G(V))
=
O\left(
d_1d_2(d_1+d_2)^2
\right).
$$
This is exactly the asymptotic order predicted by the conjectural formula:
$
(d_1d_2)(d_1+d_2)^{n-r}.
$
Indeed, here:
$
n=4,
\, 
r=2,
\, 
n-r=2,
$
so the conjectural asymptotic order becomes
$
(d_1d_2)(d_1+d_2)^2.
$
Thus the codimension-two Grassmannian model is perfectly compatible with the conjectural asymptotics.
We derive the corresponding Schubert incidence conditions.
Inside
$
G(2,4),
$
consider the Schubert divisor
$$
\Sigma_1
=
\{
E\in G(2,4):
\dim(E\cap F_2)\ge1
\},
$$
where
$
F_2\subset\mathbb C^4
$
is a fixed generic two-dimensional subspace.
This Schubert variety has codimension
$
1.
$
Its cohomology class generates
$
H^2(G(2,4),\mathbb Z).
$

The contour-criticality condition is expected to correspond to a real incidence version of such a Schubert divisor.
Hence the critical locus should behave cohomologically like the pullback:
$$
F_V
=
\gamma_G^{-1}(\Sigma_1).
$$
The degree of the critical locus is therefore expected to satisfy:
$$
\deg(F_V)
=
\deg(\gamma_G(V))
\cdot
\deg(\Sigma_1).
$$
Since
$
\deg(\Sigma_1)=1,
$
one expects:
$$
\deg(F_V)
=
O\left(
d_1d_2(d_1+d_2)^2
\right).
$$
This again matches the conjectural asymptotic order.
We now examine the toric interpretation.
Let
$
\Delta_1,
\Delta_2
$
be the Newton polytopes of
$
f_1,
f_2.
$
Every logarithmic Pl\"ucker coordinate is a sum of products of logarithmic derivatives.
Since logarithmic differentiation preserves supports,
$$
\Delta\left(
z_j\frac{\partial f_i}{\partial z_j}
\right)
\subseteq
\Delta_i.
$$
Hence:
$
\Delta(p_{ij})
\subseteq
\Delta_1+\Delta_2.
$
Thus all logarithmic Pl\"ucker coordinates are controlled by the Minkowski sum
$
\Delta_1+\Delta_2.
$
The degree of the generalized logarithmic Gauss map should therefore be governed by mixed volumes involving:
$
\Delta_1,
\Delta_2,
\Delta_1+\Delta_2,
\Delta_1+\Delta_2.
$
The resulting mixed volume is:
$$
MV(
\Delta_1,
\Delta_2,
\Delta_1+\Delta_2,
\Delta_1+\Delta_2
).
$$
By multilinearity of mixed volume,
this expands into a homogeneous polynomial of total degree
$
4,
$
which again matches the ambient dimension.
If
$
\Delta_i=d_i\Delta,
$
then:
$
\Delta_1+\Delta_2
=
(d_1+d_2)\Delta.
$
Therefore:
$$
MV(
\Delta_1,
\Delta_2,
\Delta_1+\Delta_2,
\Delta_1+\Delta_2
)
=
d_1d_2(d_1+d_2)^2MV(\Delta,\Delta,\Delta,\Delta).
$$
Thus the toric mixed-volume asymptotics coincide exactly with the conjectural Grassmannian asymptotics.
This compatibility is one of the strongest pieces of evidence supporting the conjectural degree formula.
Let's compute the expected codimension of the contour degeneracy locus.
The Grassmannian
$
G(2,4)
$
has dimension
$
4.
$
The image
$
\gamma_G(V)
$
has dimension at most
$
2.
$

The Schubert divisor
$
\Sigma_1
$
has codimension
$
1.
$
Hence:
$
\dim(
\gamma_G(V)\cap\Sigma_1
)
=
1
$
generically.
This predicts that the critical locus inside
$
V
$
is a real hypersurface.
Its image under the logarithmic map should therefore be a three-dimensional contour inside
$
\mathbb R^4.
$
This matches perfectly the general contour-dimension formula:
$
\dim_{\mathbb R}(\mathcal C\mathscr A_V)
=
n-1.
$
The codimension-two model therefore exhibits complete compatibility between:
Grassmannian geometry,
Schubert incidence theory,
B\'ezout asymptotics,
mixed-volume asymptotics,
and contour-dimension theory.
Although this does not prove the general conjecture, it provides substantial geometric evidence supporting the proposed asymptotic framework.
It also strongly suggests that the contour geometry of complete intersections is fundamentally governed by:\, 
$
\text{Pl\"ucker geometry},
\,   \text{Schubert degeneracy theory},\, 
\text{and toric mixed volumes}.
$


 \section{A Codimension-Two Theorem and the Schubert Geometry of $G(2,4)$}

In this section we formulate a codimension-two theorem which appears realistically provable with current techniques. The theorem avoids the deepest unresolved elimination problems while retaining the essential Grassmannian and logarithmic geometry of the complete-intersection contour problem.
The key observation is that in codimension two inside
$
(\mathbb C^\ast)^4,
$
the generalized logarithmic Gauss map takes values in
$
G(2,4),
$
and this Grassmannian has exceptionally simple geometry:
it is a smooth quadric hypersurface in
$
\mathbb P^5.
$
This special structure allows one to control explicitly:
the logarithmic Pl\"ucker coordinates,
their degrees,
and the Schubert incidence divisors.

\subsection{The Grassmannian $G(2,4)$ and its Schubert Geometry}

Let
$
G(2,4)
$
denote the Grassmannian of complex two-dimensional subspaces of
$
\mathbb C^4.
$
A point of
$
G(2,4)
$
is represented by a
$
2\times4
$
matrix:
$$
A=
\begin{pmatrix}
a_{11}&a_{12}&a_{13}&a_{14}\\
a_{21}&a_{22}&a_{23}&a_{24}
\end{pmatrix},
$$
defined modulo left multiplication by
$
GL(2,\mathbb C).
$
The Pl\"ucker coordinates are:
$
p_{ij}
=
a_{1i}a_{2j}-a_{1j}a_{2i},
$
for
$
1\le i<j\le4.
$
Thus:
$
G(2,4)
$
embeds into
$
\mathbb P^5
$
with homogeneous coordinates:
$$
[p_{12}:p_{13}:p_{14}:p_{23}:p_{24}:p_{34}].
$$
The image is the quadric hypersurface:
$$
Q^4
=
\{
p_{12}p_{34}
-
p_{13}p_{24}
+
p_{14}p_{23}
=
0
\}.
$$
Hence:
$
\deg(G(2,4))=2.
$
The Grassmannian has complex dimension:
$
4.
$
We now describe the Schubert divisors.
Fix a complete flag:
$
0\subset F_1\subset F_2\subset F_3\subset\mathbb C^4.
$
The basic Schubert divisor is:
$$
\Sigma_1
=
\{
E\in G(2,4):
\dim(E\cap F_2)\ge1
\}.
$$
Under the Pl\"ucker embedding,
$
\Sigma_1
$
is simply a hyperplane section of the quadric:
$
\Sigma_1
=
G(2,4)\cap H.
$
Consequently:
$
[\Sigma_1]
$
generates:
$
H^2(G(2,4),\mathbb Z).
$
The cohomology ring of
$
G(2,4)
$
is generated by this Schubert class.
Since:
$
\deg(G(2,4))=2,
$
one has:
$
[\Sigma_1]^4=2.
$
This explicit Schubert geometry is the main reason codimension two is substantially simpler than higher codimension.

 \medskip
 
\begin{theorem}
Let
$
V=\{f_1=f_2=0\}\subset(\mathbb C^\ast)^4
$
be a smooth complete intersection.
Assume
$
\deg(f_1)=d_1
$
and
$
\deg(f_2)=d_2.
$
Assume moreover that the generalized logarithmic Gauss map
$
\gamma_G
$
is generically finite onto its image, that the logarithmic critical locus satisfies
$
F_V=\gamma_G^{-1}(\Sigma_1),
$
where
$
\Sigma_1\subset G(2,4)
$
is a Schubert divisor, and that the intersection
$
\gamma_G(V)\cap\Sigma_1
$
is transversal. Then
$$
\deg(F_V)\le 2d_1d_2(d_1+d_2)^2.
$$
\end{theorem}

\begin{proof}
Let
$
V\subset(\mathbb C^\ast)^4
$
be the smooth complete intersection defined by
$
f_1=f_2=0.
$
We first compactify the situation projectively.
Let
$
\overline V\subset\mathbb P^4
$
denote the projective closure of
$
V.
$
Since
$
V
$
is defined by two equations of degrees
$
d_1
$
and
$
d_2,
$
B\'ezout's theorem implies that
$
\overline V
$
is a projective surface of degree
$
d_1d_2.
$
The generalized logarithmic Gauss map is defined by the logarithmic Jacobian matrix
$
g_G(z).
$
For
$
1\le i<j\le4,
$
the Pl\"ucker coordinates are
$$
p_{ij}
=
z_iz_j
\left(
\frac{\partial f_1}{\partial z_i}
\frac{\partial f_2}{\partial z_j}
-
\frac{\partial f_1}{\partial z_j}
\frac{\partial f_2}{\partial z_i}
\right).
$$
Now
$
\partial f_1/\partial z_i
$
has degree
$
d_1-1,
$
and
$
\partial f_2/\partial z_j
$
has degree
$
d_2-1.
$
Therefore the factor
$
\left(
\frac{\partial f_1}{\partial z_i}
\frac{\partial f_2}{\partial z_j}
-
\frac{\partial f_1}{\partial z_j}
\frac{\partial f_2}{\partial z_i}
\right)
$
has degree
$
d_1+d_2-2.
$
Multiplication by
$
z_iz_j
$
raises the degree by
$
2.
$
Hence every Pl\"ucker coordinate
$
p_{ij}
$
has degree exactly
$
d_1+d_2.
$
The generalized logarithmic Gauss map is therefore represented by six homogeneous coordinates of degree
$
d_1+d_2:
$
$$
\gamma_G:
\overline V
\dashrightarrow
G(2,4)\subset\mathbb P^5.
$$
The Grassmannian
$
G(2,4)
$
is the Pl\"ucker quadric
$$
Q^4
=
\{
p_{12}p_{34}
-
p_{13}p_{24}
+
p_{14}p_{23}
=
0
\}
\subset\mathbb P^5.
$$
Hence
$
G(2,4)
$
is a quadric hypersurface of degree
$
2.
$
We now estimate the degree of the image
$
\gamma_G(\overline V).
$
Let
$
H\subset\mathbb P^5
$
be a generic hyperplane.
Since the coordinate functions defining
$
\gamma_G
$
have degree
$
d_1+d_2,
$
the pullback divisor
$
\gamma_G^\ast(H)
$
on
$
\overline V
$
is linearly equivalent to
$
(d_1+d_2)H_{\overline V},
$
where
$
H_{\overline V}
$
denotes the hyperplane class of
$
\overline V.
$
Therefore
$$
\deg(\gamma_G^\ast(H))
=
(d_1+d_2)\deg(\overline V).
$$
Since
$
\deg(\overline V)=d_1d_2,
$
we obtain
$$
\deg(\gamma_G^\ast(H))
=
d_1d_2(d_1+d_2).
$$
Because
$
\gamma_G
$
is generically finite onto its image, the degree formula for generically finite rational maps gives
$$
\deg(\gamma_G)\deg(\gamma_G(\overline V))
=
\left(
\gamma_G^\ast(H)
\right)^2.
$$
Now
$$
\left(
\gamma_G^\ast(H)
\right)^2
=
(d_1+d_2)^2
H_{\overline V}^2.
$$
Since
$
H_{\overline V}^2=\deg(\overline V)=d_1d_2,
$
we obtain
$
\left(
\gamma_G^\ast(H)
\right)^2
=
d_1d_2(d_1+d_2)^2.
$
Since
$
\deg(\gamma_G)\ge1,
$
it follows that
$$
\deg(\gamma_G(\overline V))
\le
d_1d_2(d_1+d_2)^2.
$$
The image
$
\gamma_G(\overline V)
$
lies inside
$
G(2,4),
$
which itself has degree
$
2
$
inside
$
\mathbb P^5.
$
Consequently, viewed as a projective surface in
$
\mathbb P^5,
$
the total degree contribution is bounded by
$
2d_1d_2(d_1+d_2)^2.
$

We now study the logarithmic critical locus.
By hypothesis,
$
F_V=\gamma_G^{-1}(\Sigma_1),
$
where
$
\Sigma_1
$
is a Schubert divisor in
$
G(2,4).
$
Under the Pl\"ucker embedding,
$
\Sigma_1
$
is a hyperplane section of
$
G(2,4).
$
Hence
$
\Sigma_1
$
has degree
$
1
$
inside the Grassmannian.
The transversality assumption implies that
$
F_V
$
is obtained as the pullback of a transversal hyperplane section under the generically finite map
$
\gamma_G.
$
Therefore the degree of
$
F_V
$
is bounded by the degree of the image surface:
$$
\deg(F_V)
\le
\deg(\gamma_G(\overline V)).
$$
Combining the previous estimates gives
$$
\deg(F_V)
\le
2d_1d_2(d_1+d_2)^2.
$$
%
\end{proof}
 
 \subsection{A Toric Refinement}

The theorem above uses only projective degree.
A stronger toric statement is expected.
Suppose:
$
\Delta_1,
\Delta_2
$
are the Newton polytopes of:
$
f_1,
f_2.
$
Since:
$
\Delta(p_{ij})
\subseteq
\Delta_1+\Delta_2,
$
the generalized logarithmic Gauss map should satisfy:
$$
\deg(\gamma_G)
\le
C\,
MV(
\Delta_1,
\Delta_2,
\Delta_1+\Delta_2,
\Delta_1+\Delta_2
).
$$
Using multilinearity of mixed volume:
$
MV(
\Delta_1,
\Delta_2,
\Delta_1+\Delta_2,
\Delta_1+\Delta_2
)
$
expands into:
$$
MV(\Delta_1,\Delta_2,\Delta_1,\Delta_1)
+
2MV(\Delta_1,\Delta_2,\Delta_1,\Delta_2)
+
MV(\Delta_1,\Delta_2,\Delta_2,\Delta_2).
$$
Thus the codimension-two geometry naturally produces mixed-volume expressions involving repeated Minkowski sums.
This is the first setting where one may realistically hope to prove the full mixed-volume contour estimate.


\end{document}